\documentclass[review,letterpaper]{elsarticle}
\usepackage[english]{babel}
\usepackage[utf8]{inputenc}
\usepackage{geometry}
\usepackage{comment}
\usepackage{multicol,float}
\usepackage{color}
\usepackage{longtable,multirow}
\usepackage{array}
\usepackage{amsmath}

\usepackage{amsfonts}
\usepackage{amssymb}
\usepackage{graphicx}
\usepackage{cancel}
\usepackage{cancel}
\usepackage[all]{xy}
\usepackage{xcolor}
\usepackage{subcaption}             
\usepackage{pifont}
\newcommand{\cmark}{\ding{51}}%
\newcommand{\xmark}{\ding{55}}%
\usepackage{tikz}
\usetikzlibrary{fit, arrows.meta, positioning}

\usepackage{color,soul}
\usepackage{cases}
\usepackage{caption}
\usepackage{subcaption}
\usepackage{mathtools}
\usepackage{framed}
\usepackage{rotating}
\usepackage[strict]{changepage}
\usepackage[normalem]{ulem} 
\usepackage{rotating} 
\usepackage{natbib}
\usepackage{mathtools}
\usepackage{hyperref}

\newtheorem{propo}{Proposition}
\newtheorem{coro}{Corolary}
\newtheorem{lema}{Lemma}
\newtheorem{rem}{Remark}

\newcommand{\proof}{{\noindent \bf Proof}:\ }

\title{A strong formulation for Multiple Allocation Hub Location based on supermodular inequalities}
\author[uca]{Elena Fern\'{a}ndez}
\author[uca]{Nicol\'{a}s Zerega}
\address[uca]{Statistics and Operations Research Department, Universidad de Cádiz, Puerto Real, Spain}
\date{}
\begin{document}
	\begin{abstract}
		We introduce a new formulation for the multiple allocation hub location problem that exploits supermodular properties and uses 1- and 2-index variables only. We show that the new formulation produces the same Linear Programming bound as the tightest existing formulations for the studied problem, which use 4-index variables, outperforming existing supermodular formulations adapted to the considered problem. Computational results are presented with instances of up to 200 nodes optimally solved within a time limit of two hours.
	\end{abstract}
	
	\maketitle
	
	\section{Introduction}\label{sec:intro}
		Hub location problems (HLPs) are challenging combinatorial optimization problems that integrate strategic facility location and network design decisions, with operational routing decisions. These problems typically arise in systems where flows are consolidated and distributed through intermediate transfer facilities known as hubs. Prominent examples include the airline industry, where major airports function as hubs connected to smaller airports via feeder flights; telecommunications networks, where switches redistribute data loads; and logistics applications such as freight cargo and last-mile delivery.
		
		Formally introduced by \citet{Okelly86}, these problems involve installing \textit{hubs} that serve as consolidation, sorting, and redistribution facilities for servicing demand, which is expressed by a given set of flows that must be routed between origin-destination pairs (commodities) of an existing network. Consolidation at hubs generates economies of scale, leading to lower transportation unit costs compared to direct point-to-point systems.
		
		Broadly speaking HLPs involve two types of decisions. On the one hand deciding the set of hubs to be activated and the interhub connections (backbone network); on the other hand, determining the paths for routing the commodities through the backbone network.
		
		In this work we consider an HLP with a multiple allocation (MA) policy, according to which non-hub nodes may be connected to multiple hubs, allowing flows originated at (or with destination to) a given non-hub node to be routed through different hubs. MA-HLPs are among the most studied problems in this area and some \textit{fundamental} HLPs belong to this class, namely the so called MA-$p$HLP and MA-HLP. In the MA-$p$HLP a fixed number $p$ of hubs must be activated and the objective is the minimization of the routing of the commodities. In contrast, in the MA-HLP the number of hubs to activate is not fixed, but each activated hub incurs a setup cost. Its objective is the minimization of the total setup costs plus the routing costs. Both models share an optimality property, according to which, there is an optimal solution where the backbone network is the complete graph induced by the activated hubs. According to this property, there is an optimal solution in which the routing path of each commodity consists of at most three legs: one access arc, one interhub arc, and one distribution arc. This property has been intensively exploited in the formulations proposed in the literature.
				
		Over the years, various formulations for MA-HLPs have been proposed. \citet{Campbell94} developed the first linear formulations. Subsequent improvements were proposed by \citet{Skorin96}, \citet{Hamacher04} and \citet{Marin-et-al06}. All these formulations use path-based variables, that is, 4-index variables which explicitly model the route between an origin-destination pair. While such models produce tight linear programming (LP) bounds, they suffer from scalability issues, due to their large number of variables ($\mathcal{O}(n^4)$). To address this, \citet{Ernst98} introduced a 3-index flow-based formulation for the MA-$p$HLP. This formulation produces a weaker LP bound than its 4-index counterpart, although it uses fewer variables ($\mathcal{O}(n^3)$), allowing larger instances to be solved.
				
		Formulations that use 2-index variables have also been developed for the MA-$p$HLP and the MA-HLP. \citet{Garcia2} proposed a radius-based formulation for the MA-$p$HLP, with routing costs defined by a precomputed sorted set, presenting results with instances of up to 200 nodes. Finally, \citet{ContrerasFernandez14} introduced a formulation based on supermodular inequalities for a large class of MA-HLPs, which can be particularized for the MA-HLP and the MA-$p$HLP. For the latter, this formulation is equivalent to the formulation by \citet{Garcia2}.
		
		In this paper we focus on improving the formulation by \citet{ContrerasFernandez14} for the MA-HLP. We reinforce the supermodular inequalities and demonstrate that the resulting formulation achieves the LP bounds of the tightest path-based formulations of \citet{Hamacher04} and \citet{Marin-et-al06}. Figure \ref{fig:lit_rev} provides a visual overview of the relationships among these formulations and our contribution.
		
		The remainder of this paper is structured as follows. Section \ref{sec:notation} introduces the notation and provides a formal problem definition of the MA-HLP. Section \ref{sec:lit_formulations} reviews existing path-based formulations and those presented by \citet{ContrerasFernandez14}. In Section \ref{sec:reinfore-super}, we present our reinforced formulation, which uses 2-index variables and supermodular constraints. We analyze its properties, showing that it is as tight as the tightest path-based formulation. Computational results and discussions are presented in Section \ref{sec:computational}, solving instances of up to 200 nodes. Finally, conclusions and future research directions are outlined in Section \ref{sec:conclusions}.
		
		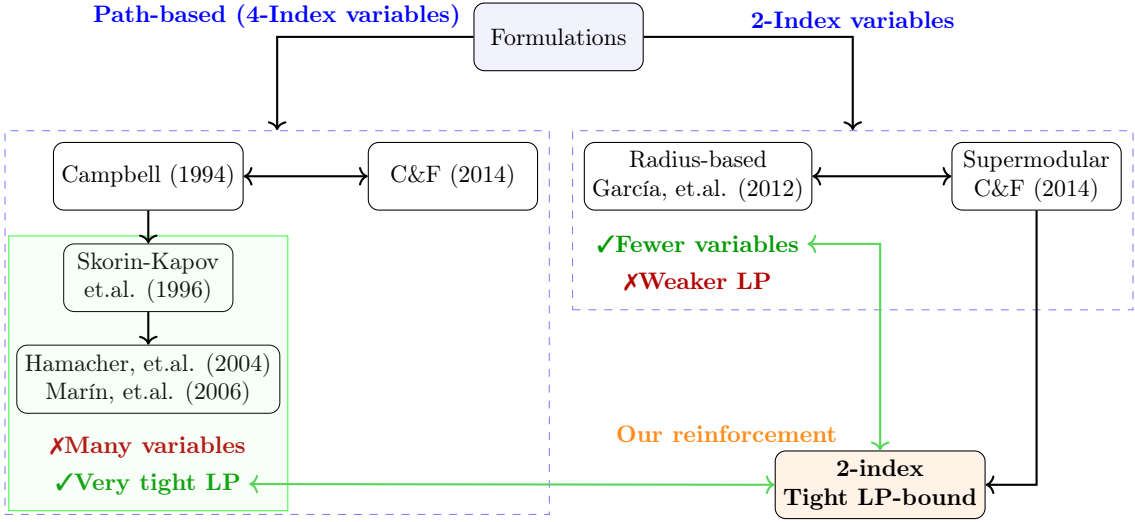
\begin{figure}[H]
			\centering
			\begin{tikzpicture}[
	scale=0.90, every node/.style={transform shape},
	node distance=1.5cm,
	box/.style={draw, rounded corners, align=center, minimum width=2.5cm, minimum height=1cm},
	arrow/.style={->, thick},
	arc/.style={<->, thick},
	pro/.style={text=green!60!black, font=\normalsize},
	con/.style={text=red!70!black, font=\normalsize},
	]
	
	\node[box, fill=blue!5] (formulations) {Formulations};
	
	\node[box, below left of = formulations, xshift=-5cm, yshift=-1.0cm] (campbell) {Campbell (1994)};
	\node[box, below of = campbell] (skorin) {Skorin-Kapov\\et.al. (1996)};
	\node[box, below of = skorin] (hamacher) {Hamacher, et.al. (2004) \\ Marín, et.al. (2006)};
	\node[box, right of = campbell, xshift=3.0cm] (cf) {C\&F (2014)};
	
	\node[con, below of = hamacher, yshift=.50cm](pathcon) {\xmark\textbf{Many variables}};
	\node[pro, below of = pathcon, yshift=.95cm](pathpro) {\cmark\textbf{Very tight LP}};
	
	\node[draw=blue!50, dashed, fit=(campbell) (cf) (skorin) (hamacher) (pathcon) (pathpro), inner sep=0.15cm] (pathgroup) {};
	
	\node[draw=green!75, fill=green!25, fill opacity=0.10, fit = (skorin) (hamacher) (pathpro), inner sep=0.10cm] (pathgrouptight) {};
	
	\node[box, below right of = formulations, xshift=1.0cm, yshift=-1.0cm] (garcia) {Radius-based\\ García, et.al. (2012) };
	\node[box, right of = garcia, xshift=3.5cm] (cfs) {Supermodular\\ C\&F (2014)};
	\node[box, below left of= cfs, yshift=-3.5cm, xshift=-1.25cm, fill=orange!10, label={[orange!90] above left:\textbf{Our reinforcement}}] (ours) {\textbf{2-index}\\ \textbf{Tight LP-bound}};
	
	\node[pro, below of = garcia, yshift=.50cm](twopro) {\cmark\textbf{Fewer variables}};
	\node[con, below of = twopro, yshift=.95cm](twocon) {\xmark\textbf{Weaker LP}};
	
	\node[draw=blue!50, dashed, fit= (garcia) (cfs) (twopro) (twocon), inner sep=0.15cm](twogroup){};
	
	\draw[arrow] (formulations.west) -| (pathgroup.north) node[midway, above, blue] {\textbf{Path-based (4-Index variables)}};
	\draw[arrow] (campbell.south) -- (skorin.north);
	\draw[arc] (campbell.east) -- (cf);
	\draw[arrow] (skorin.south) -- (hamacher.north);
	
	\draw[arrow] (formulations.east) -| (twogroup.north) node[midway, above, blue] {\textbf{2-Index variables}};
	\draw[arc] (garcia.east) -- (cfs.west);
	\draw[arrow] (cfs.south) |- (ours.east);

	\draw[<->, green!75!black!65, thick] (pathpro.east) -- (ours.west);
	\draw[<->, green!75!black!65, thick] (twopro.east) -| (ours.north);
	
\end{tikzpicture}
			\caption{Relationship between the relevant formulations and our contribution}
			\label{fig:lit_rev}
		\end{figure}
	
	\section{Notation and problem definition}\label{sec:notation}
		Consider a complete network $N=(V, A)$, with node set $V=\{1, 2,\dots, n\}$ and arc set $A=\{(i, j): i, j\in V\}$.  Associated with each arc $(i,j)\in A$, there is a unit routing cost, denoted by $c_{ij}$. We assume that routing costs satisfy the triangle inequality. Let $E=\{ij: (i, j)\in A \text{ or } (j, i)\in A, i< j\}$ denote the set of (undirected) edges underlying the arc set $A$.
				
		We assume {that every node $i\in V$ is a potential hub location with an associated setup cost $f_i$.} Any edge {$ij\in E$} connecting a pair of potential hubs can be activated as an \textit{interhub edge} (or just \textit{hub edge}). {Note that since all nodes are potential hubs, the set of potential interhub edges is precisely $E$}. These assumptions can be made without loss of generality since arbitrarily large routing costs $c_{ij}$ can be assigned to non-existing arcs and arbitrarily large  setup costs $f_i$ can be assigned to nodes that are not potential hub nodes. To alleviate notation, in the remainder of this paper any edge $ij\in E$ with $i, j\in V$, $i<j$ will be indistinctively denoted by $ji$.
				
		Service demand is given by a set of commodities defined over pairs of users, indexed by a set $R$. {Each commodity $r\in R$ is defined by a triplet $(o^r,d^r,w^r)$, where a demand $w^r\geq 0$} must be routed from origin node $o^r\in V$ to destination node $d^r\in V$. When the context is clear, we refer to the origin and destination of a commodity $r$ simply as $o$ and $d$, respectively. Without loss of generality, we assume the graph induced by commodities with $w^r>0$ is connected; otherwise, the problem could be decomposed into independent subproblems for each connected component.
		
		Finally, the set of edges incident with a given node $i\in V$, is denoted by $\delta(i)=\{e\in E: e=ij, \text{ or } e=ji \text{ with } j\in V\}$.
		
		\subsection{Problem definition}
			We study the {MA-HLP}, in which any feasible solution is given by a set of activated hub nodes, a set of activated interhub edges, and a set of  routing paths, one for each commodity. The two endnodes of any activated interhub edge must be activated hubs as well. The routing  path of a given commodity  $r\in R$ must be of the form $o^r - i_1 - \dots - i_t- d^r$ where: $(a)$ every intermediate node is activated as a hub node  and $(b)$ each pair of consecutive intermediate nodes is connected by an activated interhub edge. The routing cost of commodity $r\in R$ through a given routing path is $w^r\left(\gamma c_{o^rk_1}+\alpha \sum_{i=1}^{t-1} c_{k_i k_{i+1}}+ \theta c_{k_td^r}\right)$, where $0\le \alpha\leq 1$ is a given interhub discount cost factor and $\gamma, \theta >\alpha$ are factors applied {to the access (origin to first hub) and distribution (last hub to destination)} arcs, respectively.
					
			Since we assume that the input graph is complete, costs satisfy the triangle inequality, and any edge can be activated as an interhub edge incurring no setup cost, we will only consider routing paths with at most one interhub edge, since any intermediate subpath $i_1 - \dots - i_t$ can be substituted by a single interhub arc {$(i_1,i_t)$} with a routing cost not greater than that of the subpath. Hence, in the following, we will assume that routing paths are of the form $o^r - i - j - d^r$, where, $i,j$ are activated hubs, and $ij\in E$ is an activated interhub edge. The routing cost for commodity $r\in R$ on this path is ${C_{rij}}=w^{r}\left(\gamma c_{o^ri}+\alpha c_{ij}+\theta c_{jd^r}\right)$. If $o^r$ is itself a hub, then  $i=o^r$ and consider the first leg $(o^r,i)$ as a fictitious arc. Similarly, if $d^r$ is a hub, then $j=d^r$ and consider the last leg $(j,d^r)$ as a fictitious arc.
	
	\section{Existing formulations for the $MA-HLP$}\label{sec:lit_formulations}		
		In this section we present some existing formulations for the MA-HLP, relevant to the work that we develop. Initially, we focus on traditional formulations in which the routing of the commodities is determined by 4-index path variables \cite{Skorin96,Hamacher04,Marin-et-al06}. Some of these formulations are celebrated because they produce very tight LP bounds. Then, we discuss the adaptation to the MA-HLP of the two formulations introduced in \citet{ContrerasFernandez14}, which both produce the same LP bound. The first one uses 4-index path variables, whereas the other one is based on supermodular properties of HLPs and uses 1- and 2-index decision variables associated with the activated hubs and interhub links, respectively.
		
		\subsection{Formulations with path variables}
		\citet{Skorin96} provided formulations for several variants of the MA-HLP. All formulations use one set of binary decision variables $z_i\in\{0, 1\}$, $i\in V$ for the hubs that are activated, plus one additional set of continuous \textit{path variables} $X_{rij}\geq 0$, which indicate the fraction of the flow associated with commodity $r\in R$ that is routed through the interhub arc $(i, j)\in A$. Then, for the $MA-HLP$, the formulation of \citet{Skorin96}  is the following:
		\begin{subequations}\label{SK}
			\begin{align}
				(SK)\qquad  \min & \sum_{i\in V}f_i z_i\,+ \sum_{r\in  R}\sum_{i,j\in V} {C_{rij}}{X_{rij}}&& \label{SK-of}\\
				\mbox{s.t. }
				& \sum_{i,j\in V}{X_{rij}}=1 &&   r\in  R \label{SK:route-all}\\
				& \sum_{j\in V}{X_{rij}}\leq z_i && r\in  R, i\in V \label{SK-X-zi}\\
				& \sum_{i\in V}{X_{rij}}\leq z_j && r\in  R, i\in V \label{SK-X-zj}\\
				& z_{i}\in\{0, 1\},\, i\in V;\, {X_{rij}}\geq 0,\, r\in  R,\, i,j\in V.\label{MCL:domain}
			\end{align}
		\end{subequations}
		
		Although variables ${X_{rij}}$ are defined as continuous, it is well-known that there is an optimal solution to the MA-HLP where these variables take binary values. Hence these variables can be interpreted as the indicator of the routing path used for the different commodities. By constraints \eqref{SK-X-zi}-\eqref{SK-X-zj} the two intermediate nodes in routing paths must be activated hubs. Indeed, when $i=j$ there is just one single intermediate hub.
		
		\citet{Hamacher04} and \citet{Marin-et-al06} developed reinforcements of $SK$ in which the sets of constraints \eqref{SK-X-zi}-\eqref{SK-X-zj} are \textit{merged} in one single set of constraints:
		
		\begin{subequations}\label{MCL}
			\begin{align}
				(HLP_{MA})\qquad  \min & \sum_{i\in V}f_i z_i\,+ \sum_{r\in  R}\sum_{i,j\in V} {C_{rij}}{X_{rij}}&& \tag{\ref{SK-of}}\\
				\mbox{s.t. }
				& \sum_{i,j\in V}{X_{rij}}=1\, &&   r\in  R \tag{\ref{SK:route-all}}\\
				& {X_{rii}}+ \sum_{j\in V\setminus \{i\}}\left({X_{rij}}+{X_{rji}}\right) \leq z_i && r\in  R, i\in V \label{MCL:X-zi-zj}\\
				& z_{i}\in\{0, 1\},\, i\in V;\, {X_{rij}}\geq 0,\, r\in  R,\, i,j\in V. \tag{\ref{MCL:domain}}
			\end{align}
		\end{subequations}
		The main difference between the formulations of \citet{Hamacher04} and \citet{Marin-et-al06} is that the later uses exactly the same decision variables as $SK$, whereas the former, for a given $r\in  R$, uses one single decision variable ${X_{rij}}$ for each pair $i, j\in V$, $i\ne j$, with a routing cost ${F_{rij}}=\min\{{C_{rij}}, {C_{rji}}\}$.
				
		The idea of determining \textit{a priori} the best routing cost of a given commodity $r\in R$ \textit{through} a pair of intermediate nodes $i, j\in V$ is further exploited in the  formulations introduced in \citet{ContrerasFernandez14}. These formulations are very general, in the sense that they comprise as special cases many classes of {uncapacitated MA-HLPs,} including not only the MA-HLP but also problems with setup costs for the activated interhub edges, problems with cardinality constraints for both the number of activated hubs and the number of activated interhub edges,  and hub-arc location problems. Below we adapt the formulations of \cite{ContrerasFernandez14} to the MA-HLP, which allows us to compare these formulations and our extensions with $HLP_{MA}$.		
		
		\citet{ContrerasFernandez14} use the value $\overline F_{re}=\min\{F_{re},\, H_{ri},\ H_{rj}\}$ for the optimal cost for routing commodity $r\in R$ \emph{via} interhub edge $e=ij\in E$, where, as indicated above, $F_{er}=\min\{C_{rij}, C_{rji}\}$, is the routing cost of the best traversal of edge $ij$ and $H_{ri}=c_{o^ri}+c_{id^r}$  and $H_{rj}=c_{o^rj}+c_{jd^r}$ are the routing costs through single hubs $i$ and $j$, respectively.
		
		The decision variables of the path formulation of \citet{ContrerasFernandez14} are the following:
		\begin{itemize}
			\item $y_{e}\in\{0,1\}$ for all $e\in E$. $y_e=1$ if and only if edge $e$ is activated as an interhub edge.
			\item $z_i\in\{0, 1\}$ for all $i\in V$. $z_i=1$ if and only if a hub is activated at node $i$
			\item ${x_{re}}\in\{0, 1\}$ for all $e\in E$, $r\in  R$. ${x_{re}}=1$ if and only iff commodity $r$ is routed using its \textit{best} path through interhub edge $e$ (as determined by the value ${\overline F_{re}}$).
			
			Similarly to the above variables $\mathbf X$, even if the nature of variables $\mathbf{x}$  is discrete, in formulation $CF-P$ they can be relaxed, as it is usual with path variables in the absence of  capacity constraints (see, e.g., \citet{Hamacher04,Marin-et-al06,Contrerasetal11,ContrerasDiaz11}).
		\end{itemize}
		
		The adaptation to the MA-HLP of the path formulation of \citet{ContrerasFernandez14} simply eliminates from the objective function the term corresponding to the cost of the activated interhub edges, as well as the cardinality constraints on the number of activated hubs and interhub edges. It is as follows:
		
		\begin{subequations}
			\begin{align}
				(CF-P)\qquad  \min & \sum_{i\in V}f_iz_i\, + \sum_{r\in  R}\sum_{e\in E}{\overline F_{re}}{x_{re}}&& \label{of0}\\
				\mbox{s.t. }
				& \sum_{e\in E}{x_{re}}=1\, &&   r\in  R\label{const:route-all}\\
				& {x_{re}}\leq y_e\, &&   r\in  R,\, e\in E\label{const:route-edge}\\
				& y_e\leq z_{i}\, &&     e=ij\in E\label{const:y-z1}\\
				& y_e\leq z_{j}\, &&     e=ij\in E\label{const:y-z2}\\
				& y_e\in\{0, 1\}, e\in E; z_{i}\in\{0, 1\},\, i\in V;\, {x_{re}}\geq 0,\, r\in  R, e\in E.\label{domain-CF}
			\end{align}
		\end{subequations}
		
		\begin{rem}\label{Remark1}
			Note that formulation $CF-P$ is only valid for instances whose optimal solution activates at least one interhub edge (i.e., two hubs). By constraints \eqref{const:route-all}-\eqref{const:route-edge}, each commodity is routed \textit{via} one activated edge, even if its routing path does not traverse the edge and only one of its two endnodes is used. Still, in such a case, constraints \eqref{const:y-z1}-\eqref{const:y-z2} force that both endnodes are activated as hubs.
			
			\begin{figure}[H]
				\begin{subfigure}[t]{0.5\textwidth}
					\centering
					\includegraphics[width=0.5\textwidth]{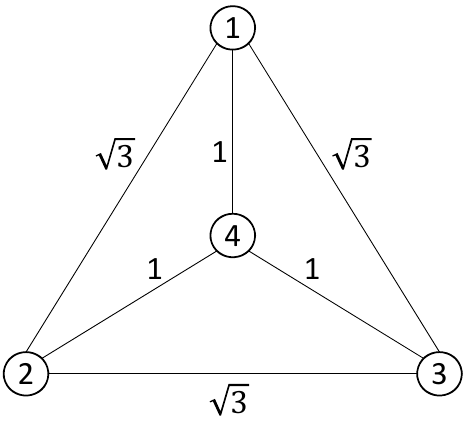}
					\caption{Instance for which $CF-P$ does not produce an optimal\\ MA-HLP solution}\label{Fig1}
				\end{subfigure}
				\begin{subfigure}[t]{0.5\textwidth}
					\centering
					\includegraphics[width=0.5\textwidth]{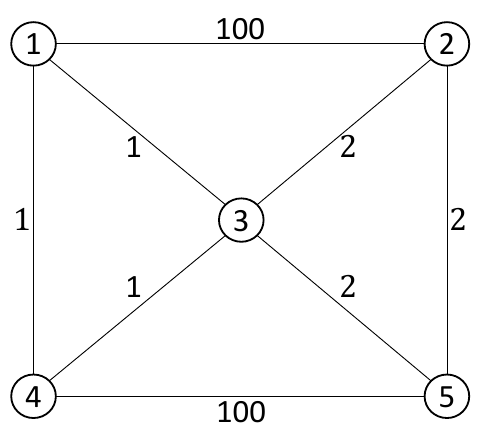}
					\caption{Instance for which $HLP_{MA}$ produces a tighter LP\\ bound than $CF-P$}\label{Fig2}
				\end{subfigure}
				\caption{Networks illustrating Remark \ref{Remark1}}
			\end{figure}
			
			This happens, for instance, in the example shown in Figure \ref{Fig1} with set of commodities $R=\{(o, d): o,d\in V, o<d\}$, $w^r=1$, for all $r\in R$, $f_i=10$, for all $i\in V$, edge costs $c_{ij}$ shown next to the edges, and $\alpha=1$.

			In the following, unless otherwise stated, we assume that we consider instances whose optimal solution activates at least two hubs. This is the case with most of the HLP instances available in the literature. Note that for such instances $CF-P$ is a valid formulation. Given that optimal solutions activate at least two hubs, and it involves no activation costs for interhub edges, $CF-P$  will have an optimal solution in which all edges connecting a pair of activated hubs are activated (at least one activated interhub edge).
		\end{rem}
				
		{In general,} $HLP_{MA}$ produces tighter LP bounds than $CF-P$. {This follows from the fact that} any feasible solution to the LP relaxation of  $CF-P$ can be expressed in terms of a feasible solution to the LP relaxation of $HLP_{MA}$, with ${x_{re}}={X_{rij}}+{X_{rji}}+{X_{rii}}+{X_{rjj}}$, while the reverse is not true. As a consequence, the LP bound of $CF-P$ can be rather weak in comparison to that produced by $HLP_{MA}$. This is illustrated in the instance of Figure \ref{Fig2}, for which the set of commodities $R=\{(o, d): o,d\in V, o<d\}$, $w^r=1$, for all $r\in R$, and $f_i=1$, for $i\in \{2, 3, 5\}$ and $f_i=100$, otherwise. The values next to depicted edges indicate their routing costs $c_{ij}$ with non-depicted edges having arbitrarily large routing costs.		
		For this instance, the LP relaxation of $HLP_{MA}$ produces an integer optimal solution of value 24, where $z_2=z_3=1$ and $z_i=0$ otherwise. The path variables values are ${X_{r33}}=1$, for all $r\in R\setminus\{(2, 5)\}$, $X_{2522}=1$, and ${X_{rij}}=0$ otherwise. On the contrary, the LP relaxation of $CF-P$ produces {a fractional solution with an objective value of 23.5. This solution has $z_2=z_3=z_5=0.5$ (and $z_1=z_4=0$), with ${x_{re}}=0.5$ for $e\in \{23,\,35\}$ and ${x_{re}}=0$ otherwise for all $r\in R$.}
				
		The poor performance of the LP relaxation of $CF-P$ can be attributed to the following. When, for $r\in R$, $\bar e\in E$ given, the value $\overline F_{\bar er}$ corresponds to a path that uses one of the endnodes of $\bar e$, say $i_u$, as its single intermediate vertex, i.e. $\overline F_{\bar er}=H_{i_ur}$,  then it may happen that there exist other edges also incident with $i_u$, $e\in\delta(i_u)$, $e\ne \bar e$, with the same routing value, i.e., such that $\overline F_{er}=\overline F_{\bar er}$. In such a case, formulation $CF-P$ may produce solutions such that $x_{er}=\varepsilon$, for all $e\in\delta(i_u)$ such that $\overline F_{er}=\overline F_{\bar er}$ and $z_{i_u}=\varepsilon$.

		For the example of Figure \ref{Fig2}, this happens for commodity $r=(1,4)$. For this commodity, edges $23$ and $35$ produce the same single-hub optimal path via node 3, which has a routing cost of ${\overline F_{r23}}={\overline F_{r35}} = {H_{r3}} = 2$. This is reflected in the LP solution where $x_{r23}$ = $x_{r35}=z_3=\varepsilon=0.5.$
		
		\subsection{A reinforcement of $CF-P$}
			Below we develop a reinforcement of $CF-P$, that will be denoted as $FZ-P$, which produces the same LP bound as $HLP_{MA}$. In Section \ref{sec:reinfore-super}, we will see how to derive a formulation based on \textit{supermodular} constraints that produces the same LP bound as $FZ-P$, and thus as $HLP_{MA}$, while retaining the 1- and 2-index decision variables only.
			
			Indeed, formulation $CF-P$ can be reinforced by substituting constraints \eqref{const:y-z1}-\eqref{const:y-z2} by the following ones, which mimic the rationale of \eqref{MCL:X-zi-zj}:
			\begin{subequations}
			\begin{align}
			& \sum_{e\in\delta(i)}{{x_{re}}}\leq z_i\, &&   r\in  R,\, i\in V.\label{const:route-edge+}
			\end{align}
			\end{subequations}
			
			The resulting formulation is:
			
			\begin{subequations}
			\begin{align}
			(FZ-P)\qquad  \min & \sum_{i\in V}f_iz_i\, + \sum_{r\in  R}\sum_{e\in E} {\overline F_{re}}{{x_{re}}}&& \tag{\ref{of0}}\\
			\mbox{s.t. }
			& \sum_{e\in E}{{x_{re}}}=1\, &&   r\in  R\tag{\ref{const:route-all}}\\
			& {{x_{re}}}\leq y_e\, &&   r\in  R,\, e\in E\tag{\ref{const:route-edge}}\\
			& \sum_{e\in\delta(i)}{{x_{re}}}\leq z_i, &&   r\in  R,\, i\in V\tag{\ref{const:route-edge+}}\\
			& y_e\in\{0, 1\}, e\in E; z_{i}\in\{0, 1\},\, i\in V;\, {{x_{re}}}\geq 0\, r\in  R, e\in E.\tag{\ref{domain-CF}}
			\end{align}
			\end{subequations}
			
			As  we next prove, formulation $FZ-P$ is essentially the same as $HLP_{MA}$. Before, we introduce some additional notation and analyze some properties that will be useful to relate $HLP_{MA}$ and  $FZ-P$.
			
			For $r\in  R$, let
			\begin{itemize}
			\item $E^r=\{e=ij\in E:  {\overline F_{re}}={F_{re}}<\min\{{H_{ri}}, {H_{rj}}\}\}$.
			
			When ${F_{re}}\geq \min\{{H_{ri}}, {H_{rj}}\}$, there is an optimal solution to the MA-HLP where commodity $r$ is not routed using interhub edge $e$, as it would be at least as good to route it using a single-hub path through $i$ or $j$. According to this definition we assume that ties are broken in favor of single-hub paths. Thus, for the optimal routing of a given commodity $r\in  R$, the only edges that need to be considered are those indexed in $E^r$.
			
			\item $V^r{\subseteq V}$ the
			set of {nodes} that can be the only intermediate hub in an optimal routing path for commodity $r$. {A node $i$ belongs to $V^r$ if an optimal path for $r$ would route $w^r$ directly through $i$ rather than traversing an activated edge $e(i,r)\in \delta(i)$, thus} $V^r=\{i{\in V}: \exists e(i,r){\in\delta(i) \text{ s.t. } H_{ri} < F_{re}}\}$. With very few exceptions $V^r=V$. More precisely, let $U^r=\{i\in V: c_{id^r} >c_{jd^r}  \text{ for all } j\in V\}$. That is, $U^r$ is the set of nodes $i\in V$ such that the unit routing of the distribution arc $(i, d^r)$ is higher than the unit routing cost of any other distribution arc. Thus, when $i\in U^r$ the distribution arc $(i, d^r)$ is the most expensive distribution arc for commodity $r$. We will see that $V^r=V\setminus U^r$.
			
			\end{itemize}
			We can now formally state the following properties of formulation $HLP_{MA}$ that will be useful to compare it with $FZ-P$.
			\begin{lema}\label{lemita}
			There is an optimal solution to $HLP_{MA}$, $({\mathbf{z}^*}, {\mathbf{X}}^*)$, such that, for all $r\in R$,
			\begin{itemize}
			\item[$(i)$] ${X^{*}_{rij}}+ {X^{*}_{rji}}=0$ for all $e\in E\setminus E^r$.
			
			\item[$(ii)$] {For all $i\in V\setminus U^r$} there exists some edge, $e(i,r)=i\bar{\jmath}\,{\in \delta(i)}$, {such that} ${X^{*}_{ri\bar{\jmath}}}=0$, i.e., $e(i,r)$ would not be traversed from $i$ to $\bar \jmath$ in an optimal routing of commodity $r$, even if it were activated.
			
			\item[$(iii)$] ${X^{*}_{rii}}=0$ for all $i\in U^r$.
			\end{itemize}
			\end{lema}
			\proof
			\begin{itemize}
			\item[$(i)$] Follows directly from the definition of $E^r$.
			
			\item[$(ii)$]  {Let $i\in V\setminus U^r$. Then, there exists $\bar{\jmath}\in V$ such that $c_{i d^r}\leq c_{\bar{\jmath}d^r}$. Then, $C_{ri\bar{\jmath}}=c_{o^ri}+\alpha c_{i\bar{\jmath}} + c_{\bar{\jmath} d^r}\geq c_{o^r i}+\alpha c_{i\bar \jmath} + c_{i d^r}\geq c_{o^r i}+ c_{i d^r} =H_{r i}$, so $e(i,r)=i\bar \jmath$ will not be traversed from $i$ to $\bar \jmath$ in an optimal routing of commodity $r$, even if it were activated.}

			\item[$(iii)$] Let now {$i\in U^r$}. Then, by definition of $U^r$, it holds that, for all $j\in V$, $c_{id^r}>c_{jd^r}$. Hence
			\[H_{ri}=c_{o^ri}+c_{id^r}>c_{o^ri}+c_{jd^r}\geq c_{o^ri}+\alpha c_{ij}+c_{jd^r}= C_{rij},\quad \text{for all }j\in V\]
			and the result follows since the the single-hub routing of commodity $r$ through hub $i$ is worse than the routing path through interhub arc $(i,j)$ for all $j\in V$. Hence, in any solution to $HLP_{MA}$, ${X^{*}_{rii}}=0$ for all $i\in U^r$. \hfill$\blacksquare$
			
			\end{itemize}
			
			We thus conclude that $V^r=V\setminus U^r$. In the following, for all $i\in V^r$ we denote by $e(i,r)\in \delta(i)$ an arbitrarily chosen edge that would not be traversed from $i$ to $\bar{\jmath}$ in the routing of commodity $r$, even if it were activated.
			
			\begin{rem}
			Note that Lemma \ref{lemita} also applies to the LP relaxation of $HLP_{MA}$.
			\end{rem}
			
			\begin{propo}
			The  LP bound of $FZ-P$ is the same as that of $HLP_{MA}$.
			\end{propo}
			
			\begin{proof}
			We show that for every optimal solution to the LP relaxation of $HLP_{MA}$ there exists a feasible solution to the LP relaxation of $FZ-P$ with the same objective function and \textit{vice versa}.
			
			\begin{itemize}
			\item Let $(\overline{\mathbf{z}}, \overline{\mathbf{X}})$ be an optimal solution to the LP relaxation of $HLP_{MA}$. Consider now the solution $(\overline{\mathbf z}, \overline{\mathbf y}, \overline{\mathbf x})$ with the same $\overline{\mathbf{z}}$ components and $\overline y_e=\max\{ {\overline X_{rij}}+ { \overline X_{rji}}: r\in  R\}$ for all $e=ij\in E$. For, $r\in R$, $e\in E$, let also,
			\[
			{\overline x_{re}}=
			\begin{cases}
			{\overline X_{rij}} + { \overline X_{rji}} & \text{if } e\in E^r\\
			{\overline X_{rii}} & \text{if } i\in V^r,\, e=e(i,r)\\
			0 & \text{otherwise}.
			\end{cases}
			\]
			
			By construction, $(\overline{\mathbf z}, \overline{\mathbf y}, \overline{\mathbf x})$ is feasible for $FZ-P$ as it satisfies constraints \eqref{const:route-all}, \eqref{const:route-edge} and \eqref{const:route-edge+}. Moreover, its routing cost for a given commodity $r\in R$ is
			
			\[
			\sum_{e\in E}{\overline F_{re}}{\overline x_{re}}=\sum_{e\in E^r}{F_{re}}{\overline x_{re}}+\sum_{i\in V^r}{C_{rii}\overline x_{re(i,r)}}
			\]
			which, by definition of $\overline{\mathbf x}$, coincides with the objective function value of $(\overline{\mathbf{z}}, \overline{\mathbf{X}})$.
			
			\item Let $(\overline{\mathbf z}, \overline{\mathbf y}, \overline{\mathbf x})$ be an optimal solution to the LP relaxation of $FZ-P$. Without loss of generality, we assume that, ${\overline x_{re}}=0$  for all $r\in R$, $i\in V^r$,  $e\in\delta(i)\setminus E^r$, $e\ne e(i,r)$. Consider now the solution  $(\overline{\mathbf{z}}, \overline{\mathbf{X}})$ in which the $\overline{\mathbf z}$ components remain unchanged, and the components {$\overline{X}_{rij}$}, {$\overline{X}_{rji}$}, {$\overline{X}_{rii}$}, and {$\overline{X}_{rjj}$} are determined as follows:
			
			\[
			\begin{cases}
			{\overline X_{rij}}={\overline x_{re}},\, {\overline X_{rji}}={\overline X_{rii}}={\overline X_{rjj}}=0 & \text{if } e\in E^r,\, \text{ and }{\overline F_{re}}={C_{rij}},\\
			{\overline X_{rji}}={\overline x_{re}},\, {\overline X_{rij}}={\overline X_{rii}}={\overline X_{rjj}}=0 & \text{if } e\in E^r,\, \text{ and } {\overline F_{re}}={C_{rji}}<{C_{rij}},\\
			{\overline X_{rii}}={\overline x_{re}},\, {\overline X_{rij}}={\overline X_{rji}}={\overline X_{rjj}}=0 & \text{if } e=ij=e(i,r),\, \text{ with } i\in V^r,\\
			{\overline X_{rii}}={\overline X_{rij}}={\overline X_{rji}}={\overline X_{rjj}}=0 & \text{if } e=ij\notin E^r,\, e\ne e(i,r),\, \text{ with } i\in V^r,\\
			\end{cases}
			\]
			
			It is straightforward to check that $(\overline{\mathbf{z}}, \overline{\mathbf{X}})$ satisfies constraints \eqref{SK:route-all}, \eqref{MCL:X-zi-zj} and its objective function value coincides with that of $(\overline{\mathbf z}, \overline{\mathbf y}, \overline{\mathbf x})$. \hfill $\blacksquare$
			\end{itemize}
			\end{proof}

\medskip

	\section{A reinforced supermodular formulation}\label{sec:reinfore-super}
		Below we introduce an alternative formulation for the MA-HLP, which will be referred to as $FZ-S$. It takes as starting point the supermodular formulation of \citet{ContrerasFernandez14} (which will be denoted by $CF-S$) and uses exactly its same decision variables. The main advantage of $FZ-S$ over $CF-S$ is that the former produces the same LP bound as $FZ-P$ (and thus as $HLP_{MA}$), that is tighter than the LP bound of the later, which coincides with that of $CF-P$.
		
		The supermodular formulation $CF-S$ uses the same $\mathbf y$ and $\mathbf z$ variables as formulation $CF-P$, { in addition to} continuous variables $\eta^r\geq 0$ for $r\in  R$, which indicate the value of the optimal routing path for commodity $r$ in the hub network induced by $\mathbf{y}$ and $\mathbf{z}$. To guarantee that at least one interhub edge is activated, the edge set is enlarged to $E^*=E\cup\{\widetilde e\}$ where $\widetilde e$ is a fictitious edge {with an arbitrarily large routing cost} $\overline F_{\widetilde er}=M$ for all $r\in R$. 
		
		Furthermore, for each $r\in  R$, the values ${\overline F_{re}}$ are sorted by increasing order, with ties arbitrarily broken. We denote by {$e_{rt}\in E$} the edge producing the $t$-th sorted value. For simplicity, when the index $r$ is clear from the context we just write $e_t$ instead of {$e_{rt}$}. That is, ${\overline F_{re_1}}\leq {\overline F_{re_2}}\leq \dots\leq {\overline F_{re_{|E|}}}\leq {\overline F_{re_{|E|+1}}}=M$.
		
		Specifically, the supermodular formulation of \citet{ContrerasFernandez14} is:
		\begin{subequations}
			\begin{align}
				(CF-S)\qquad  \min & \sum_{i\in V}f_iz_i\, + \sum_{r\in  R}\eta^r && \label{of}\\
				\mbox{s.t. }
				& \eta^r\geq {\overline F_{re_{t}}}+\sum_{e\in E\setminus \{e_{t}\}}({\overline F_{re}}- {\overline F_{re_{t}}})^-y_{e_h}\, &&   r\in  R, {t=1,\dots,|E^*|}\label{const:valor_eta}\\ 
				& y_e\leq z_{i}\, &&     e=ij\in E^*\tag{\ref{const:y-z1}}\\
				& y_e\leq z_{j}\, &&     e=ij\in E^*\tag{\ref{const:y-z2}}\\
				& y_e\in\{0, 1\}, e\in E^*; z_{i}\in\{0, 1\},\, i\in V,\, \eta^r\geq 0,\, r\in  R,\label{domain}
			\end{align}
		\end{subequations}
		\noindent where $a^-=\min\{a, 0\}$.
		
		Constraints \eqref{const:valor_eta} have the following interpretation: the routing cost of a given commodity $r\in  R$ will be at least ${\overline F_{re_t}}$ unless some edge $e\in E^*$ with a { lower routing cost is activated}. Taking into account the sorting of the edges, constraints \eqref{const:valor_eta} can be alternatively written as
		\[
			\eta^r\geq {\overline F_{re_t}}+\sum_{h=1}^{t-1}({\overline F_{re_h}}- {\overline F_{re_{t}}})y_{e_h}\, \qquad   r\in  R, t=1,\dots, |E^*|,	
		\]
		\noindent since, for a given commodity $r\in R$ and routing value ${ \overline F_{re_t}}$, the only edges with   a routing value strictly smaller than ${ \overline F_{re_t}}$ are those sorted in position at most ${t-1}$, with $1\leq h \leq t-1$.
		
		In \citet{ContrerasFernandez14} it is proven that formulations $CF-P$ and $CF-S$ produce the same LP bound, which as explained, may not be very tight for the MA-HLP. This result is aligned with a previous one of \citet{Garcia2} for the special case of the $p$-median problem without hub activation costs, since the adaptation of formulation $CF-P$ to that special case coincides, \textit{sauf} minor modifications, with the radius formulation of \citet{Garcia2}. To the best of our knowledge, these are the only existing formulation for the MA-HLP using 1- and 2-index decision variables only.
		
		In the remainder of this section we turn our attention on how to translate the rationale of constraints \eqref{const:route-edge+} to the supermodular formulation $CF-S$, while retaining the same set of variables. For this, we will denote by ${v_{rt}}$ the $t$-th sorted routing value for commodity $r\in R$, i.e., ${v_{rt}}={\overline F_{re_t}}$, and re-write constraints \eqref{const:valor_eta} as
		\begin{subequations}
			\begin{align}
			&\eta^r\geq {v_{rt}}+\sum_{h=1}^{t-1}({\overline F_{re_h}}- {v_{rt}})y_{e_h}\, && r\in  R, t=1\dots, |E^*|.\label{const:valor_eta2}
			\end{align}
		\end{subequations}
		
		Consider a commodity $r\in R$ and indices $t, h$ with $2\leq t\leq |E^*|$ and $1\leq h \leq t-1$. Let us suppose that the $h$-th routing cost, ${v_{rh}}$, corresponds to a single-hub path through an intermediate vertex $i_h\in V^r$, i.e. ${v_{rh}}={H_{ri_h}}$. Then, $e\notin E^r$ for all $e\in\delta(i_h)$ such that ${\overline F_{re}}={v_{rh}}$, so all variables $y_e$ corresponding to such edges will have the same coefficient $\left({\overline F_{re}}-{v_{rt}}\right)<0$. In such a case, the activation of any interhub edge $e\in\delta(i_h)\setminus E^r$ with ${\overline F_{re}}={v_{rh}}$ will produce the same \textit{saving} $\left({\overline F_{re}}-{v_{rt}}\right)$. This means that  the specific edge that is activated is irrelevant, provided that ${\overline F_{re}}={v_{rh}}$. What is necessary for the feasibility of the solution is that node $i_h$ is activated as a hub.
		
		The above observation can be reflected in the right hand side of the constraints \eqref{const:valor_eta2} by  \textit{merging}, in one single term associated with decision variable $z_{i_h}$, all the addends associated with the value ${v_{rh}}$ corresponding to single-hub routing paths \textit{via} node $i_h$.
		
		For this, we will differentiate among the set of potential values for the \textit{best} routing of a given commodity $r\in R$ (relative to a given edge) according to $\{{F_{re}}\}_{e\in E^r}\cup \{{H_{ri}}\}_{i\in V^r}$. Let $\{{ \overline v_{rt}}\}_{t\in T^r}$, with $|T^r|=|E^r|+|V^r|$, denote the sequence of these values  sorted by increasing order (again ties are broken arbitrarily). Indeed, it is now necessary to control whether, for a given $t$, the value ${ \overline v_{rt}}$ corresponds to an interhub-edge path (with routing cost ${F_{re}}$, $e\in E^r$) or to a single-hub path (with routing cost ${H_{ri}}$, $i\in V^r$). For this reason, we partition $T^r=T^r_y\cup T^r_z$, where $T^r_y$ and $T^r_z$ respectively denote the sets of indices such that ${ \overline v_{rt}}$ corresponds to ${F_{re_t}}$ and to ${H_{ri_t}}$.
		
		Therefore, an alternative supermodular formulation for the MA-HLP is:
		\begin{subequations}
			\begin{align}
				(FZ-S)\qquad  \min & \sum_{i\in V}f_iz_i\, + \sum_{r\in  R}\eta^r && \tag{\ref{of}}\\
				\mbox{s.t. }
				&\eta^r\geq { \overline v_{rt}}+\sum_{h\in T^r_y: h\leq t-1}({F_{re_h}}- { \overline v_{rt}})y_{e_h}+\sum_{h\in T^r_z: h\leq t-1}({H_{ri_h}}- { \overline v_{rt}})z_{i_h}\, &&   r\in  R, t\in T^r\label{const_super-z-1}\\
				& y_e\leq z_{i}\, &&     e=ij\in E^*\tag{\ref{const:y-z1}}\\
				& y_e\leq z_{j}\, &&     e=ij\in E^*\tag{\ref{const:y-z2}}\\
				& y_e\in\{0, 1\}, e\in E^*; z_{i}\in\{0, 1\},\, i\in V,\, \eta^r\geq 0,\, r\in  R.\tag{\ref{domain}}
			\end{align}
		\end{subequations}
		
		To alleviate notation, in the following we will not specify whether the summation indices  belong to $T^r_y$ or to $T^r_z$ as this can be immediately deduced from the decision variables involved in the corresponding sum. Still, it is important to recall that all the terms involving some $\mathbf{y}$ variable  correspond to indices of $T^r_y$ (i.e. they are associated with edges of $E^r$) whereas all the terms involving some $\mathbf{z}$ variable correspond to  indices of $T^r_z$ (i.e. they are associated with nodes of $V^r$).
		
		We note that, for the special case of the $p$-median without hub activation costs, \citet{Garcia2} developed a family of valid inequalities, which can be interpreted quite similarly to \eqref{const_super-z-1}.
		
		We next analyze the LP bound produced by $FZ-S$.
		
		\begin{propo}
			Let $\overline{\mathbf{y}}\in [0, 1]$ and $\overline{\mathbf{z}}\in [0, 1]$ be optimal values for the LP relaxation of FZ-S. For $r\in R$, let $\overline t_r=\min\{t\in T^r: \sum_{h=1}^{t}\overline y_{e_h}+\sum_{h=1}^{t}\overline z_{i_h}\geq 1\}$. Then,
			\[
				\overline v_{FZS}=\sum_{i\in V}f_i\overline z_i + \sum_{r\in R}{v_{r\overline t_r}}\left(1-\sum_{h=1}^{\overline t_r-1}\overline y_{e_h}-\sum_{h=1}^{\overline t_r-1}\overline z_{i_h}\right) +\sum_{h=1}^{\overline t_r-1} {F_{re_h}}\overline  y_{e_h}+\sum_{h=1}^{\overline t_r-1}{H_{ri_h}}\overline z_{i_h}.
			\]
		\end{propo}

		\begin{proof}
			For $r\in R$, $t\in T^r$ given, let ${S_{rt}}$ denote value of the right-hand-side of the associated constraint \eqref{const_super-z-1} for the solution $\overline{\mathbf{y}}$, $\overline{\mathbf{z}}$. Then
			\begin{align*}
				&{S_{rt}}={ \overline v_{rt}}+\sum_{h=1}^{t-1}({ F_{re_h}}- { \overline v_{rt}})\overline y_{e_h}+\sum_{h=1}^{t-1}({ H_{ri_h}}- { \overline v_{rt}})\overline z_{i_h}=&&\\
				&= { \overline v_{rt}}\left(1-\sum_{h=1}^{t-1}\overline y_{e_h}-\sum_{h=1}^{t-1}\overline z_{i_h}\right)+ \sum_{h=1}^{t-1}{ F_{re_h}} \overline y_{e_h}+ \sum_{h=1}^{t-1}{ H_{ri_h}} \overline z_{i_h}.
			\end{align*}
			
			Let us see that ${S_{r\overline t_r}}\geq { S_{rt}}$ both for $1\leq t < \overline t_r$ and for $\overline t_r <t$. Indeed:
			\begin{itemize}
			\item ${S_{r\overline t_r}}\geq {S_{rt}}$ for $1\leq t < \overline t_r$.
			
				\begin{align*}
					&{S_{r\overline t_r}}={v_{r\overline t_r}}\left(1-\sum_{h=1}^{\overline t_r-1}\overline y_{e_h}-\sum_{h=1}^{\overline t_r-1}\overline z_{i_h}\right)+ \sum_{h=1}^{\overline t_r-1}{ F_{re_h}} \overline y_{e_h}+ \sum_{h=1}^{\overline t_r-1}{ H_{ri_h}} \overline z_{i_h}\geq &&\\
					&{ \overline v_{rt}}\left(1-\sum_{h=1}^{\overline t_r-1}\overline y_{e_h}-\sum_{h=1}^{\overline t_r-1}\overline z_{i_h}\right)+ \sum_{h=1}^{\overline t_r-1}{ F_{re_h}} \overline y_{e_h}+ \sum_{h=1}^{\overline t_r-1}{ H_{ri_h}} \overline z_{i_h}=&&\\
					&{ \overline v_{rt}}\left(1-\sum_{h=1}^{t-1}{\overline y_{e_h}}-\sum_{h=t}^{\overline t_r-1}{\overline y_{e_h}}-\sum_{h=1}^{t-1}\overline z_{i_h}-\sum_{h=t}^{\overline t_r-1}\overline z_{i_h}\right)+&&\\ &\qquad\qquad \sum_{h=1}^{t-1}{ F_{re_h}} \overline y_{e_h}+ \sum_{h=t}^{\overline t_r-1}{ F_{re_h}} \overline y_{e_h}+ \sum_{h=1}^{t-1}{ H_{ri_h}} \overline z_{i_h}+ \sum_{h=t}^{\overline t_r-1}{ H_{ri_h}} \overline z_{i_h}=&&\\
					&={S_{rt}} +\sum_{h=t}^{\overline t_r-1}\left({ F_{re_h}}-{ \overline v_{rt}}\right)\overline y_{e_h}+\sum_{h=t}^{\overline t_r-1}\left({ H_{ri_h}}-{ \overline v_{rt}}\right)\overline z_{i_h}\,\geq {S_{rt}}.
				\end{align*}
				
				\noindent The first inequality holds since ${v_{r\overline t_r}}\geq { \overline v_{rt}}$ and, by the definition of $\overline t_r$, $1-\sum_{h=1}^{\overline t_r-1}\overline y_{e_h}-\sum_{h=1}^{\overline t_r-1}\overline z_{i_h}\geq 0$. The second inequality follows since, for all $h\geq t$, ${ F_{re_h}}-{ \overline v_{rt}}\geq 0$ and ${ H_{ri_h}}-{ \overline v_{rt}}\geq 0$, and also $\overline y_{e_h}\geq 0$ and $\overline z_{i_h}\geq 0$.
			\item ${S_{r\overline t_r}}\geq {S_{rt}}$ for $\overline t_r <t$.
			
				\begin{align*}
					&{S_{rt}}={ \overline v_{rt}}\left(1-\sum_{h=1}^{t-1}\overline y_{e_h}-\sum_{h=1}^{t-1}\overline z_{i_h}\right)+ \sum_{h=1}^{t-1}{ F_{re_h}} \overline y_{e_h}+ \sum_{h=1}^{t-1}{ H_{ri_h}} \overline z_{i_h} = &&\\
					&= { \overline v_{rt}}\left(1-\sum_{h=1}^{\overline t_r-1}\overline y_{e_h}-\sum_{h=1}^{\overline t_r-1}\overline z_{i_h}\right)-{ \overline v_{rt}}\left(\sum_{h=\overline t_r}^{t-1}\overline y_{e_h}+\sum_{h=\overline t_r}^{t-1}\overline z_{i_h}\right)+\\
					&\qquad\qquad + \left[\sum_{h=1}^{\overline t_r-1}{ F_{re_h}} \overline y_{e_h}+\sum_{h=1}^{\overline t_r-1}{ H_{ri_h}} \overline z_{i_h}\right]+ \left[\sum_{h=\overline t_r}^{t-1}{ F_{re_h}} \overline y_{e_h}+ \sum_{h=\overline t_r}^{t-1}{ H_{ri_h}} \overline z_{i_h}\right]\leq &&\\
					&\leq {v_{r\overline t_r}}\left(1-\sum_{h=1}^{\overline t_r-1}\overline y_{e_h}-\sum_{h=1}^{\overline t_r-1}\overline z_{i_h}\right)-{ \overline v_{rt}}\left(\sum_{h=\overline t_r}^{t-1}\overline y_{e_h}+\sum_{h=\overline t_r}^{t-1}\overline z_{i_h}\right)+\\
					&\qquad\qquad + \left[\sum_{h=1}^{\overline t_r-1}{ F_{re_h}} \overline y_{e_h}+ \sum_{h=1}^{\overline t_r-1}{ H_{ri_h}} \overline z_{i_h}\right]+ \left[\sum_{h=\overline t_r}^{t-1}{ F_{re_h}} \overline y_{e_h}+ \sum_{h=\overline t_r}^{t-1}{ H_{ri_h}} \overline z_{i_h}\right]= &&\\
					= &{S_{r\overline t_r}}+ \sum_{h=\overline t_r}^{t-1}\left({ F_{re_h}}-{ \overline v_{rt}}\right)\overline y_{e_h}+\sum_{h=\overline t_r}^{t-1}\left({ H_{ri_h}}-{ \overline v_{rt}}\right)\overline z_{i_h}\leq {S_{r\overline t_r}}.
				\end{align*}
				
				\noindent The first inequality holds since $t>t_r$ and, by the definition of $\overline t_r$, it holds that ${v_{r\overline t_r}}\leq { \overline v_{rt}}$ and $1-\sum_{h=1}^{t-1}\overline y_{e_h}-\sum_{h=1}^{t-1}\overline z_{i_h}\leq 0$. The second inequality follows  from ${ F_{re_h}}-{ \overline v_{rt}}\leq 0$ and ${ H_{ri_h}}-{ \overline v_{rt}}\leq 0$ for all $h< t$, because $\overline y_{e_h}\geq 0$ and $\overline z_{i_h}\geq 0$.
				\end{itemize}
				
				Therefore, for a given $r\in R$, the largest right-hand side of the constraints \eqref{const_super-z-1} is attained for the index $\overline t_r$ as defined above. Hence, the objective function value of the LP of $FZ-S$ is:
				\begin{align*}
					&\overline v_{FZS}= \sum_{i\in V}f_i\overline z_i+\sum_{r\in R}{S_{r\overline t_r}}=&&\\
					&=\sum_{i\in V}f_i\overline z_i+\sum_{r\in R}{v_{r\overline t_r}}\left(1-\sum_{h=1}^{\overline t_r-1}\overline y_{e_h}-\sum_{h=1}^{\overline t_r-1}\overline z_{i_h}\right)+\sum_{h=1}^{\overline t_r-1} { F_{re_h}}\overline y_{e_h}+ \sum_{h=1}^{\overline t_r-1}{ H_{ri_h}}\overline z_{i_h}
				\end{align*}
				\noindent and the result follows. \hfill $\blacksquare$
			\end{proof}
		
		\subsection{Relating $FZ-S$ and $FZ-P$}
			Observe that, even if formulation $FZ-S$ does not include explicit variables representing the routing of the commodities, optimal routing paths can be immediately derived from optimal values for variables $\mathbf{z}$ and $\mathbf{y}$. In particular:
			
			\begin{lema}\label{lema2}
				Let ${\mathbf{z^*}}$ and ${\mathbf{y}^*}$ optimal values for variables $\mathbf{z}$ and $\mathbf{y}$ for $FZ-S$. 	For all $r\in R$, let also $t^*_r=\min\{\min\{t\in T^r_y:\, \sum_{h=1}^{t}y^*_{e_h}= 1\},\, \min\{t\in T^r_z:\, \sum_{h=1}^{t} z^*_{i_h}= 1\}\}$. Then, for all $r\in R$, optimal values
				for the routing variables of $FZ-P$ are given by:
				\[
					{x^{*}_{re}}=
					\begin{cases}
						1 & \text{ with } e=e_{t^*_r}, \text{if } t^*_r\in T^r_y, \\
						1 & \text{ with } e=e(i, r),\, i=i_{t^*_r},\,  \text{if } t^*_r\in T^r_z,\\
						0 & \text{otherwise}.
					\end{cases}
				\]
			\end{lema}
			
			\begin{proof}
				Let ${\eta^r}^*$, $r\in R$ denote the optimal values of the $\mathbf{\eta}$ variables associated with ${\mathbf{z^*}}$ and ${\mathbf{y}^*}$. Then,
				\[
					\sum_{r\in R}{\eta^r}^*=\sum_{r\in R: t^*_r\in T^r_y} {F_{re_{t^*_r}}} + \sum_{r\in R: t^*_r\in T^r_z}{H_{ri_{t^*_r}}}= \sum_{r\in R}\sum_{e\in E}{\overline F_{re}}{x^{*}_{re}},
				\]
				so the solution $\mathbf{x}^*$ as defined above, together with ${\mathbf{z^*}}$ and ${\mathbf{y}^*}$ is optimal for $FZ-P$. \hfill $\blacksquare$
			\end{proof}
			
			As we see below, the sorting and indexing used for the routing costs in formulation $FZ-S$, can also be exploited to obtain optimal values for the routing variables of the LP relaxation of $FZ-P$, from optimal values of variables $\mathbf{z}$ and $\mathbf{y}$ for the LP relaxation of $FZ-S$.
			
			\begin{lema}\label{lema3}
				Let $\overline{\mathbf{z}}$ and $\overline{\mathbf{y}}$ be optimal values for variables $\mathbf{z}$ and $\mathbf{y}$ for the LP relaxation of $FZ-P$. For all $r\in R$, let also $\overline t_r=\min\{t\in T^r: \sum_{h=1}^{t}\overline y_{e_h}+\sum_{h=1}^{t}\overline z_{i_h}\geq 1\}$. Then, the components of an optimal solution $\overline{\mathbf{x}}$ for the LP relaxation of $FZ-P$ is given by:
				
				\begin{equation}
					{\overline x_{re}}=
					\begin{cases}
						\overline y_{e} & \text{if } e={e_{h}},\, h\in T^r_y,\, h<\overline t_r\\
						\overline z_{i} & \text{if } e=e(i,r),\, i=i_h\in V^r,\,  h\in T^r_z,\, h<\overline t_r\\
						1-\sum_{h=1}^{t_r-1} {\overline y_{e_h}}-\sum_{h=1}^{t_r-1}\overline z_{i_h}& \text{if } e={e_{\overline t_r}}\\
						0 & \text{otherwise}.
					\end{cases}\label{sol-LP}
				\end{equation}
			\end{lema}
			
			\begin{proof}
				Observe that the solution $\overline{\mathbf{x}}$ is built over the support graph induced by the (fractional) solution $\overline{\mathbf{z}}$ and $\overline{\mathbf{y}}$, where potential routing paths can be identified. The indices $h\in T$ of the activated decision variables allow to determine, whether for a given $r\in R$, such paths involve edges in $E^r$ or just one intermediate hub. In the later case ($h\in T^r_z$) a single edge $e(i_h,r)$ incident with the involved node is associated with the single-hub path. This guarantees that \eqref{const:route-edge+} are satisfied. Moreover, for each $r\in R$, the values ${\overline X_{re}}$ are assigned in a greedy fashion by increasing values of their routing paths, until the corresponding constraint \eqref{const:route-all} is satisfied, so the total routing cost of the obtained routing paths are of minimum total cost. \hfill $\blacksquare$
			\end{proof}
			
			\begin{coro}
				$FZ-P$ and $FZ-$S produce the same LP bound.
			\end{coro}
			
			\begin{proof}
				It can be easily checked that, for all $r\in R$, the routing cost of the solution obtained according to \eqref{sol-LP} coincides with ${S_{r\overline t_r}}$. Thus, the result follows.   \hfill $\blacksquare$
			\end{proof}

	\section{Computational experiments}\label{sec:computational}
		In this section we provide comprehensive comparisons among the $HLP_{MA}$; the $CF-S$ adapted to the MA-HLP; and the $FZ-S$ formulation in terms of LP bounds, solving times, explored nodes and scalability.
		
		All the experiments were conducted on a PC equipped with a Ryzen 7 5700G CPU and 32~GB of RAM. The formulations were implemented in Python 3.11 and solved using Gurobi 10.0.1. For reproducibility, we set Gurobi's \textit{Threads} parameter to 1 and disabled the \textit{Presolve} option.
		
		There are several well known HLP datasets widely used to compare results in the literature. The ones used in the experiments are:
		\begin{itemize}
			\item Civil Aeronautics Board (CAB) dataset, containing data from 100 airports of the USA with demands between each pair of airports \citep[see,][]{Okelly87}. It provides symmetric commodity demands $w^r$ and unit routing costs $c_{ij}$. Test instances were generated for sizes $n\in \{20,\,25,\,30,\,35,\,40,\,50,\,60,\,70,\,80,\,90,\,100\}$. For $n<100$, we considered the first $n$ entries of the dataset; for $n=100$, we used the entire dataset.
			\item The Australian Post (AP) dataset, introduced by \citet{Ernst96}, contains 200 points of data with asymmetric commodity demands $w^r$ and unit routing costs $c_{ij}$. This dataset is mainly used because of its large size, therefore, the instances were generated for sizes $n\in \{100,\,125,\,150,\,175,\,200\}$. Similarly as above, the first $n$ entries of the dataset were considered for each value of $n<200$ and the entire dataset for $n=200$.
		\end{itemize}
		
		The interhub discount factor is set to $\alpha \in \{0.2,\,0.5,\,0.8\}$. The weights for access and distribution arcs were fixed to $\gamma = 1$ and $\theta = 1$, respectively, for all the experiments. The CAB and AP hub setup costs were generated by following the methodology described in \cite{Ebery00}.
		
		The supermodular constraints were separated within a branch-and-cut algorithm, with a procedure similar to that proposed by \citet{ContrerasFernandez14}. As noted by the authors, the separation of constraints \eqref{const_super-z-1} may be done in $\mathcal{O}(|E|)$ for each commodity $r\in R$.
		
		We have set a time limit of 7200 seconds on large instances (more than 100 nodes), while smaller instances were limited to 3600 seconds.
		
		Figures \ref{fig:cab_comp} and \ref{fig:ap_comp} allow to compare the LP bounds produced by $CF-S$ adapted to the MA-HLP (blue bars), and $FZ-S$ (orange bars). The obtained results clearly show that $FZ-S$ outperforms $CF-S$, with improvements of approximately 20\% on CAB instances, and 30\% on AP instances.
		
		\begin{figure}[H]
			\centering
			\includegraphics[width=.90\linewidth]{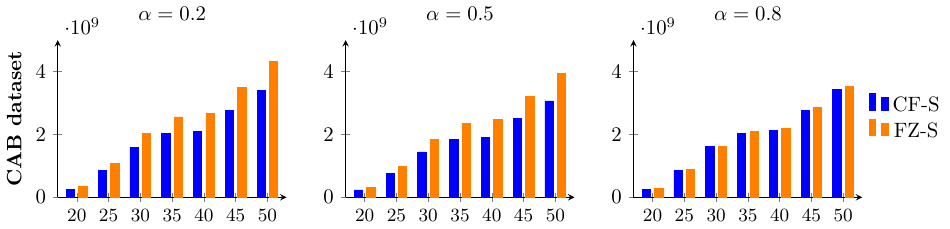}
			\caption{LP bound comparison between $CF-S$ and $FZ-S$ on CAB dataset}
			\label{fig:cab_comp}
		\end{figure}
		
		\begin{figure}[H]
			\centering
			\includegraphics[width=.90\linewidth]{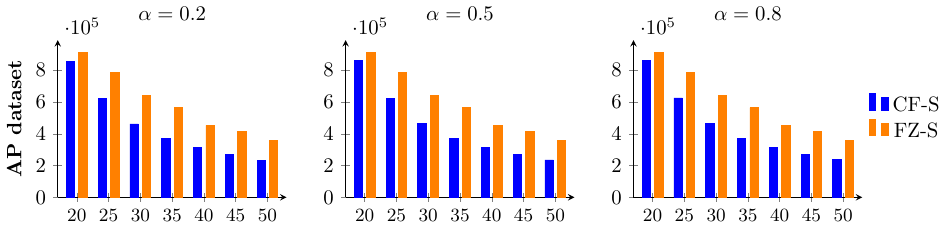}
			\caption{LP bound comparison between $CF-S$ and $FZ-S$ on AP dataset}
			\label{fig:ap_comp}
		\end{figure}
		
		Table \ref{tab:cab_comp} provides results about the overall performance of $HLP_{MA}$, $CF-S$ and $FZ-S$ on CAB instances, with up to 100 nodes. The first two columns indicate the instance size (\textit{$n$}) and interhub discount factor (\textit{$\alpha$}), respectively. The following three blocks of columns show the computing time needed to solve the instances to proven optimality (\textit{cpu (s)}) and the number of explored nodes (\textit{$\# Exp$}) for each formulation, respectively.
		
		As previously mentioned, scalability is a major issue for $HLP_{MA}$ as we were unable to solve instances larger than 50 nodes due to memory limitations. For the case of $CF-S$, in our machine we were able to solve instances of up to 70 nodes within the one hour time limit. For the optimally solved instances with $CF-S$, the number of explored nodes is, in some cases, up to 1000 times higher than with $FZ-S$, which is reflected in longer computing times.
		
		\begin{table}[H]
	\centering
	{%
		\begin{tabular}{rr|rr|rr|rr}
			\multicolumn{2}{c|}{} & \multicolumn{2}{c|}{$HLP_{MA}$} & \multicolumn{2}{c|}{$CF-S$} & \multicolumn{2}{c}{$FZ-S$} \\
			\multicolumn{1}{c}{$n$} & \multicolumn{1}{c|}{$\alpha$} & \multicolumn{1}{c}{cpu (s)} & \multicolumn{1}{c|}{$\#Exp$} & \multicolumn{1}{c}{cpu (s)} & \multicolumn{1}{c|}{$\#Exp$} & \multicolumn{1}{c}{cpu (s)} & \multicolumn{1}{c}{$\#Exp$} \\ \hline\hline 
			50 & 0.2 & 396.61 & 1 & 673.15 & 1199 & 345.83 & 1 \\
			50 & 0.5 & 351.66 & 1 & 632.10 & 786 & 182.19 & 1 \\
			50 & 0.8 & 36.22 & 1 & 500.41 & 537 & 74.67 & 1 \\
			60 & 0.2 & Mem. & - & 585.19 & 467 & 297.73 & 1 \\
			60 & 0.5 & Mem. & - & 721.89 & 1685 & 329.47 & 1 \\
			60 & 0.8 & Mem. & - & 684.48 & 349 & 120.18 & 1 \\
			70 & 0.2 & Mem. & - & 2662.35 & 931 & 1125.89 & 1 \\
			70 & 0.5 & Mem. & - & 3376.98 & 742 & 542.61 & 1 \\
			70 & 0.8 & Mem. & - & T.L. & - & 160.87 & 7 \\
			80 & 0.2 & Mem. & - & T.L. & - & 1688.71 & 1 \\
			80 & 0.5 & Mem. & - & T.L. & - & 755.88 & 1 \\
			80 & 0.8 & Mem. & - & T.L. & - & 421.28 & 1 \\
			90 & 0.2 & Mem. & - & T.L. & - & 1376.72 & 1 \\
			90 & 0.5 & Mem. & - & T.L. & - & 2442.84 & 1 \\
			90 & 0.8 & Mem. & - & T.L. & - & 1234.21 & 1 \\
			100 & 0.2 & Mem. & - & T.L. & - & 3584.07 & 1 \\
			100 & 0.5 & Mem. & - & T.L. & - & 3528.33 & 1 \\
			100 & 0.8 & Mem. & - & T.L. & - & 1143.81 & 27
		\end{tabular}%
	}
	\caption{Results of $HLP_{MA}$, $CF-S$ and $FZ-S$ on the CAB dataset}
	\label{tab:cab_comp}
\end{table}
		
		Table \ref{tab:ap_results} presents results for large AP instances, with up to 200 nodes. The first two columns are defined as in Table \ref{tab:cab_comp}. The next three columns provide the best upper bound (\textit{UB}), the best lower bound (\textit{LB}) and the lower bound at the root node (\textit{LB$_{root}$}), respectively. The following two columns report the solving time (\textit{cpu (s)}) and number of explored nodes (\textit{$\# Exp$}), respectively. The last column (\textit{Cuts}) provides the number of supermodular cuts found.
		
		Since $FZ-S$ provides the same LP relaxation than $HLP_{MA}$, most instances are solved at the root node. Instances with smaller $\alpha$ are generally harder to solve due to the burden of having many competitive interhub edges with small routing costs. As $\alpha$ increases, the number of relevant candidate edges decreases, reducing this burden. 
		
		\begin{table}[H]
	\centering
	{%
		\begin{tabular}{rr|rrrrrr}
			\multicolumn{1}{c}{$n$} & \multicolumn{1}{c|}{$\alpha$} & \multicolumn{1}{c}{$UB$} & \multicolumn{1}{c}{$LB$} & \multicolumn{1}{c}{$LB_{root}$} & \multicolumn{1}{c}{cpu (s)} & \multicolumn{1}{c}{$\#Exp$} & \multicolumn{1}{c}{Cuts} \\ \hline \hline
			100 & 0.2 & 1.36$\times 10^{5}$ & 1.36$\times 10^{5}$ & 1.36$\times 10^{5}$ & 277.68 & 1 & 19634 \\
			100 & 0.5 & 1.39$\times 10^{5}$ & 1.39$\times 10^{5}$ & 1.39$\times 10^{5}$ & 161.92 & 1 & 19862 \\
			100 & 0.8 & 1.40$\times 10^{5}$ & 1.40$\times 10^{5}$ & 1.40$\times 10^{5}$ & 67.63 & 1 & 20152 \\
			125 & 0.2 & 1.32$\times 10^{5}$ & 1.32$\times 10^{5}$ & 1.32$\times 10^{5}$ & 537.71 & 1 & 58682 \\
			125 & 0.5 & 1.37$\times 10^{5}$ & 1.37$\times 10^{5}$ & 1.37$\times 10^{5}$ & 339.52 & 1 & 64214 \\
			125 & 0.8 & 1.38$\times 10^{5}$ & 1.38$\times 10^{5}$ & 1.38$\times 10^{5}$ & 144.48 & 1 & 54840 \\
			150 & 0.2 & 1.58$\times 10^{5}$ & 1.58$\times 10^{5}$ & 1.58$\times 10^{5}$ & 1056.55 & 1 & 65700 \\
			150 & 0.5 & 1.59$\times 10^{5}$ & 1.59$\times 10^{5}$ & 1.59$\times 10^{5}$ & 791.50 & 1 & 59210 \\
			150 & 0.8 & 1.60$\times 10^{5}$ & 1.60$\times 10^{5}$ & 1.60$\times 10^{5}$ & 335.97 & 1 & 62890 \\
			175 & 0.2 & 1.43$\times 10^{5}$ & 1.43$\times 10^{5}$ & 1.43$\times 10^{5}$ & 2650.45 & 1 & 58975 \\
			175 & 0.5 & 1.50$\times 10^{5}$ & 1.50$\times 10^{5}$ & 1.50$\times 10^{5}$ & 1528.47 & 1 & 83732 \\
			175 & 0.8 & 1.51$\times 10^{5}$ & 1.51$\times 10^{5}$ & 1.51$\times 10^{5}$ & 618.12 & 1 & 86342 \\
			200 & 0.2 & 2.02$\times 10^{5}$ & 2.02$\times 10^{5}$ & 2.02$\times 10^{5}$ & 4652.08 & 1 & 98755 \\
			200 & 0.5 & 2.04$\times 10^{5}$ & 2.04$\times 10^{5}$ & 2.04$\times 10^{5}$ & 2309.08 & 1 & 112908 \\
			200 & 0.8 & 2.06$\times 10^{5}$ & 2.06$\times 10^{5}$ & 2.06$\times 10^{5}$ & 951.15 & 1 & 102666
		\end{tabular}%
	}
	\caption{Results of $FZ-S$ on large instances of the AP dataset}
	\label{tab:ap_results}
\end{table}
		
		In the experiments reported above, the obtained backbone networks consists of up to two hub nodes. This outcome can be explained by the fact that hub setup costs are significantly larger than total routing costs. 
		
		To obtain more diverse backbone networks in the larger instances, we generated instances with smaller hub setup costs. Similar to the approach by \citet{ContrerasFernandez14}, we consider a hub setup cost factor of $0.1$. Table \ref{tab:cost_factor} summarizes the results obtained with these smaller setup costs. The first two columns are defined as in Tables \ref{tab:cab_comp} and \ref{tab:ap_results}. The remaining columns are grouped in two blocks, depending on the setup cost factor. The first block corresponds to the results with the original setup costs, and the second group to the results with the new setup costs. Within each block, the columns report the objective value (\textit{val.}), number of explored nodes (\textit{$\# Exp$}), total solving time (\textit{cpu (s)}), and activated hubs (\textit{Hubs}), respectively.
		
		As expected, smaller setup costs lead to a higher number of activated hubs. Note that all instances with the same number of nodes activate the same hub nodes, independently of the discount factor $\alpha$. It is also worth noting that reducing setup costs increases the problem’s computational difficulty, as the number of competitive hub candidates increases. 
		
\begin{table}[H]
	\centering
	{%
		\begin{tabular}{rr|rrrr|rrrr}
			\multicolumn{1}{c}{} & \multicolumn{1}{c|}{} & \multicolumn{4}{c|}{Cost factor $=1$} & \multicolumn{4}{c}{Cost factor $=0.1$} \\
			\multicolumn{1}{c}{$n$} & \multicolumn{1}{c|}{$\alpha$} & \multicolumn{1}{c}{val.} & \multicolumn{1}{c}{$\# Exp$} & \multicolumn{1}{c}{$cpu\,(s)$} & \multicolumn{1}{c|}{$Hubs$} & \multicolumn{1}{c}{val.} & \multicolumn{1}{c}{$\# Exp$} & \multicolumn{1}{c}{$cpu\,(s)$} & \multicolumn{1}{c}{$Hubs$} \\ \hline \hline
			100 & 0.2 & 1.356$\times 10^5$ & 1 & 277.68 & [4, 51] & 9.802$\times 10^4$ & 1 & 343.39 & [4, 40, 51] \\
			100 & 0.5 & 1.385$\times 10^5$ & 1 & 161.92 & [4, 51] & 1.024$\times 10^5$ & 1 & 178.73 & [4, 40, 51] \\
			100 & 0.8 & 1.397$\times 10^5$ & 1 & 67.63 & [4, 51] & 1.037$\times 10^5$ & 1 & 70.65 & [4, 40, 51] \\
			125 & 0.2 & 1.323$\times 10^5$ & 1 & 537.71 & [11, 80] & 8.369$\times 10^4$ & 1 & 961.91 & [11, 46, 80] \\
			125 & 0.5 & 1.369$\times 10^5$ & 1 & 339.52 & [11, 80] & 9.059$\times 10^4$ & 1 & 515.96 & [11, 46, 80] \\
			125 & 0.8 & 1.379$\times 10^5$ & 1 & 144.48 & [11, 80] & 9.264$\times 10^4$ & 1 & 191.05 & [11, 46, 80] \\
			150 & 0.2 & 1.575$\times 10^5$ & 1 & 1056.55 & [24, 39] & 9.652$\times 10^4$ & 1 & 1974.45 & [10, 39, 52] \\
			150 & 0.5 & 1.588$\times 10^5$ & 1 & 791.50 & [24, 39] & 1.197$\times 10^5$ & 1 & 978.32 & [10, 39, 52] \\
			150 & 0.8 & 1.597$\times 10^5$ & 1 & 335.97 & [24, 39] & 1.202$\times 10^5$ & 1 & 365.99 & [10, 39, 52] \\
			175 & 0.2 & 1.432$\times 10^5$ & 1 & 2650.45 & [21, 168] & 8.847$\times 10^4$ & 1 & 2580.91 & [21, 43, 168] \\
			175 & 0.5 & 1.498$\times 10^5$ & 1 & 1528.47 & [21, 168] & 9.511$\times 10^4$ & 1 & 1545.50 & [21, 43, 168] \\
			175 & 0.8 & 1.514$\times 10^5$ & 1 & 618.12 & [21, 168] & 9.669$\times 10^4$ & 1 & 632.81 & [21, 43, 168] \\
			200 & 0.2 & 2.023$\times 10^5$ & 1 & 4652.08 & [31, 46] & 8.791$\times 10^4$ & 1 & 7144.44 & [21, 46, 52, 185] \\
			200 & 0.5 & 2.037$\times 10^5$ & 1 & 2309.08 & [31, 46] & 9.649$\times 10^4$ & 1 & 3613.36 & [21, 46, 52, 185] \\
			200 & 0.8 & 2.057$\times 10^5$ & 1 & 951.15 & [31, 46] & 9.941$\times 10^4$ & 1 & 1186.16 & [21, 46, 52, 185]
		\end{tabular}%
	}
	\caption{Results of $FZ-S$ considering the setup cost factor on large instances of the AP dataset}
	\label{tab:cost_factor}
\end{table}

	\section{Conclusions}\label{sec:conclusions}
		We have presented a new 2-index formulation for the MA-HLP that yields the same LP bounds of the tightest 4-index formulations by \citet{Hamacher04} and \citet{Marin-et-al06}.
		
		Computational experiments on instances with up to 200 nodes showed that $FZ-S$ scales more effectively than $HLP_{MA}$, and outperforms the $CF-S$ formulation adapted to the MA-HLP, solving the large instances in under two hours.
		
		The results show that our proposal effectively combines the strengths of both formulations: on the one hand, it achieves a very tight LP relaxation; on the other, it relies on a substantially smaller number of variables, improving scalability and implementation efficiency.
		
		Future work will focus on extending the formulation to handle longer paths involving more than one interhub connection. This extension would allow modeling network structures where the backbone is not necessarily complete, thereby expanding our research to more general hub network design settings.
	
	\section{Acknowledgments}
		This work was supported by the Spanish Ministry of Science and Innovation under projects \textbf{PID2023-146643NB-I00} and \textbf{PID2019-105824GB-I00}.

\bibliographystyle{abbrvnat}
\bibliography{Thesis_bib_clean2}
\end{document}